\documentclass[11pt]{amsart}

\usepackage{amsfonts,amssymb,amscd,amsmath,enumerate, verbatim,calc}
\newtheorem{thm}{Theorem}[section]
\newtheorem{cor}[thm]{Corollary}
\newtheorem{lem}[thm]{Lemma}

\theoremstyle{definition}

\newtheorem{exam}[thm]{Example}
\newtheorem{rem}[thm]{Remark}
\numberwithin{equation}{section}
\newcommand{\norm}[1]{\left\Vert#1\right\Vert}

\newcommand{\field}[1]{\mathbb{#1}}

\newcommand{\D}{\field{D}}

\begin{document}

\title[Generalized Volterra operators mapping between...]{Generalized Volterra operators mapping between\\ Banach spaces of analytic functions}

\author[Eklund] {T. Eklund}\address{Ted Eklund. Department of Mathematics, \AA bo Akademi University. FI-20500 \AA bo, Finland.
\emph{e}.mail: ted.eklund@abo.fi}
\author[Lindstr{\"o}m] {M. Lindstr\"om}\address{Mikael Lindstr{\"o}m. Department of Mathematics, \AA bo Akademi University. FI-20500 \AA bo, Finland. \emph{e}.mail: mikael.lindstrom@abo.fi} 
\author[Pirasteh] {M. M. Pirasteh}\address{Maryam M. Pirasteh. Department of Mathematics, University of Mohaghegh Ardabili, Ardabil, Iran.  \emph{e}.mail: m.pirasteh@uma.ac.ir, maryam.mpirasteh@gmail.com}
\author[Sanatpour] {A. H. Sanatpour}\address{Amir H. Sanatpour. Department of Mathematics, Kharazmi University, Tehran, Iran.  \emph{e}.mail: a$\_$sanatpour@khu.ac.ir, a.sanatpour@gmail.com}
\author[Wikman] {N. Wikman}\address{Niklas Wikman. Department of Mathematics, \AA bo Akademi University. FI-20500 \AA bo, Finland. \emph{e}.mail: niklas.wikman@abo.fi}


\begin{abstract}
We characterize boundedness and compactness of the classical Volterra operator $T_g \colon H_{v_{\alpha}}^{\infty} \to H^{\infty}$  induced by a univalent function $g$ for standard weights $v_{\alpha}$ with $0 \leq \alpha < 1$, partly answering an open problem posed by A. Anderson, M. Jovovic and W. Smith. We also study boundedness, compactness and weak compactness of the generalized Volterra operator $T_g^{\varphi}$ mapping between Banach spaces of analytic functions on the unit disc satisfying certain general conditions.
\end{abstract}

\maketitle
\section{\bf{Introduction}}
The space of analytic functions on the open unit disc  $\mathbb{D}$ in the complex plane $\mathbb{C}$ is denoted by $H(\mathbb D)$. Every analytic selfmap $\varphi \colon \mathbb D \to \mathbb D$ induces a \emph{composition operator} $C_{\varphi}f = f \circ \varphi$ on $H(\mathbb{D})$. If furthermore $u \in H(\mathbb{D})$ then one can define a \emph{weighted composition operator} $uC_{\varphi}(f) = u\cdot(f \circ \varphi)$. For general information of composition operators on classical spaces of analytic functions the reader is referred to the excellent monographs by Cowen and MacCluer \cite{18} and Shapiro \cite{19}. The main object of study in this paper is the \emph{generalized Volterra operator} $T_g^\varphi$, defined for a fixed function $g \in H(\mathbb{D})$ and selfmap $\varphi \colon \mathbb{D} \to \mathbb{D}$ as  
\begin{equation*}
T_g^{\varphi}(f)(z) = \int_{0}^{\varphi(z)}{f(\xi)g^{\prime}(\xi)d\xi},  \ \ z\in \mathbb{D}, f \in H(\mathbb{D}).
\end{equation*}
In the special case when $\varphi$ is the identity map $\varphi(z)=z$ we get the \emph{classical Volterra operator}
\begin{equation*}
T_g(f)(z) = \int_{0}^{z}{f(\xi)g^{\prime}(\xi)d\xi},  \ \ z\in \mathbb{D}, f \in H(\mathbb{D}),
\end{equation*}
which has been extensively studied on various spaces of analytic functions during the past decades, starting from the papers \cite{28}, \cite{26}, \cite{27} and  \cite{25}  by  Pommerenke, Aleman, Cima and  Siskakis.

The present paper is inspired by two recent works on boundedness and compactness properties of the Volterra operator $T_g$ mapping into the space $H^{\infty}$ of bounded analytic functions on the unit disc equipped with the supremum norm $\|\cdot\|_{\infty}$. Namely, Smith, Stolyarov and Volberg \cite{13} obtained a very nice necessary and sufficient condition  for $T_g$ to be bounded on $H^\infty$ when $g$ is univalent. The main purpose of this paper is to demonstrate that  similar conditions  characterize  boundedness and compactness of $T_g \colon H^\infty_{v_\alpha} \to H^\infty$ when  $0 \le \alpha <1$ and $g$ is univalent. This is done in section 2  which contains the main result of the paper (Theorem 2.3). In the other work,  Contreras, Pel{\'a}ez, Pommerenke and R{\"a}tty{\"a} \cite{2} studied  boundedness, compactness and weak compactness of $T_g \colon \mathcal{X} \to H^{\infty}$ acting on a Banach space $\mathcal{X} \subset H(\mathbb D)$. In section 3 we carry out a similar study of the generalized Volterra operator  $T_g^\varphi$ mapping  between Banach spaces of analytic functions on the unit disc satisfying certain general conditions. Among other things, we estimate the norm and essential norm  of $T_g^\varphi$  mapping from $\mathcal{X}$ into the spaces $H_v^{\infty}$ and $\mathcal B_v^{\infty}$ (see definitions below), and we also show that weak compactness and compactness coincide when $T_g^\varphi$ acts on $H_v^{\infty}$ or $\mathcal B_v^{\infty}$, for very general target spaces.

To introduce this general framework, let $\mathcal{X}$ be a Banach space of analytic functions on the unit disc $\mathbb D$ and let $\|\cdot\|_{\mathcal{X}}$ denote its norm. For any $z\in \mathbb D$, the evaluation functional $\delta_z\colon \mathcal{X} \to \mathbb{C}$  is defined as $\delta_z(f) = f(z)$ for
$f \in \mathcal{X}$. Throughout this paper we will assume that $\mathcal{X}$ contains the constant functions, and hence all $\delta_z$ are non-zero. We will also consider the following conditions on the space $\mathcal{X}$ (see \cite{12} or \cite{5}):

\begin{itemize}
\item[(I)] The  closed unit ball $B_{\mathcal{X}}$ of $\mathcal{X}$ is compact with respect to  the compact open topology $co$. In particular,  the identity map  $\text{id} \colon (\mathcal{X} , \norm{\cdot}_{\mathcal{X}}) \to (\mathcal{X} ,co)$ is continuous and hence $\delta_z\in \mathcal{X}^*.$
\item[(II)] The evaluation functionals $\delta_z\colon \mathcal{X} \to \mathbb{C}$ satisfy $\lim_{|z|\to 1} \|\delta_z\|_{\mathcal{X} \to\,\mathbb{C}} = \infty$.
\item[(III)] The linear operator $T_r\colon \mathcal{X} \to \mathcal{X}$ mapping $f\mapsto f_r$, where $f_r(z) = f(rz)$, is compact for every $0 < r < 1$.
\item[(IV)] The operators $T_r$ in (III) satisfy $\sup_{0< r<1} \norm{T_r}_{\mathcal{X} \to \mathcal{X}}< \infty$.
\item[(V)] The pointwise multiplication operator $M_u\colon \mathcal{X} \to \mathcal{X}$ satisf{}ies
$\norm{M_u}_{\mathcal{X} \to \mathcal{X}} \lesssim \norm{u}_{\infty}$ for every $u\in H^\infty$, and in particular $H^\infty\subset\mathcal{X}$.
\end{itemize}
We use the notation $A\lesssim B$ to indicate that there is a positive constant $c,$ not depending on properties of $A$ and $B,$ such that $A\le cB.$ We will also write $A\asymp B$ whenever both $A\lesssim B$ and $B\lesssim A$ hold. 

The evaluation map $Q(f) = \widehat{f} \colon \mathcal{X} \to \mathcal{X}^{\ast\ast}$, where $\widehat{f}(\ell) = \ell(f)$ for $f \in \mathcal{X}$ and $\ell \in {\mathcal{X}}^{\ast}$, serves as a natural embedding of any Banach space $\mathcal{X}$ into its second dual. From condition (I) we obtain by using the Dixmier-Ng theorem \cite{6} that the space
\begin{equation*}
^*\!{\mathcal{X}}:= \{\ell \in \mathcal{X}^*: \ell|B_{\mathcal{X}} \text{ is $co$-continuous}\},
\end{equation*}
endowed with the norm induced by the dual space $\mathcal{X}^*$, is a Banach space and that the evaluation map $\Phi_{\mathcal{X}} \colon \mathcal{X} \to (^*\!{\mathcal{X}})^*$, defined as the restriction $\Phi_{\mathcal{X}}(f) = \widehat{f}\,\big|{^*\!{\mathcal{X}}}$, is an onto isometric isomorphism. In particular, $^*\!{\mathcal{X}}$ is a predual of $\mathcal{X}.$  Moreover, it follows from the Hahn-Banach theorem that the linear span of the set $\{\delta_z : z \in\mathbb D\}$ is contained and norm dense in $^*\!{\mathcal{X}}$, see \cite{7} for more details.

Among the spaces that will be considered in this paper are the \emph{weighted Banach spaces of analytic functions}  $H_v^{\infty}$ and $H_v^0$ given by
\begin{align*}
H_v^{\infty}  &= \ \Big\{f\in H(\D)\,:\, \norm{f}_{H^\infty_v}:= \sup_{z\in \D}v(z)|f(z)| < \infty\Big\},\\
H_v^0 \ &= \ \Big\{f\in H_v^{\infty}\,:\,\lim_{|z| \rightarrow 1^-}v(z)|f(z)| = 0\Big\},
\end{align*}
where the \emph{weight} $v \colon \mathbb{D} \rightarrow \mathbb{R}$ is a continuous and strictly positive function such that $\lim_{|z|\to 1} v(z) =0$. The weight $v$ is called \emph{normal} if it is \emph{radial}, in the sense that $v(z)= v(|z|)$ for every $z \in \mathbb{D}$, non-increasing with respect to $|z|$ and satisfies  
\begin{equation*}
\inf_{k \in \mathbb{N}} \limsup_{n \to \infty} \frac{v(1-2^{-n-k})}{v(1-2^{-n})} < 1 \ \ \text{ and } \ \ \sup_{n \in \mathbb{N}} \frac{v(1-2^{-n})}{v(1-2^{-n-1})} < \infty.
\end{equation*}
Lusky \cite{4} has shown  that $H_v^{\infty}\approx \ell^\infty$ and $H_v^0\approx c_0$ for a large class of  weights including the normal weights. By $\mathcal{X} \approx \mathcal{Y}$ we mean that the spaces $\mathcal{X}$ and $\mathcal{Y}$ are isomorphic. The \emph{standard weights} $v_{\alpha}(z) := (1-|z|^2)^{\alpha}$ with $\alpha > 0$ are clearly normal, and for these we sometimes write $H_{\alpha}^{\infty}$ instead of $H_{v_{\alpha}}^{\infty}$ and $H_{\alpha}^0$ instead of $H_{v_{\alpha}}^0$, with similar pattern for other types of spaces. 

The \emph{Bloch-type spaces} $\mathcal B_v^{\infty}$ and $\mathcal B_v^0$  are defined by
\begin{align*}
\mathcal B_v^{\infty}  &= \ \Big\{f\in H(\D)\,:\,  \norm{f}_{\mathcal B_v^{\infty}}:= |f(0)| + \sup_{z\in \D}v(z)|f^{\prime}(z)| < \infty\Big\},\\
\mathcal B_v^0 \ &= \ \Big\{f\in \mathcal B_v^{\infty}\,:\,  \lim_{|z|\rightarrow 1^-}v(z)|f^{\prime}(z)| = 0\Big\},
\end{align*}
and we also denote $\widetilde{\mathcal B}_v^{\infty} := \{f \in \mathcal B_v^{\infty}:  f(0)=0 \}$ and $\widetilde{\mathcal B}_v^0 := \{f \in \mathcal B_v^0:   f(0)=0 \}$. If the weight $v$ is normal then, by a result of Lusky \cite{17} and using the weight $w(z) = (1-|z|)v(z)$, one can identify $H_{v}^{\infty} = \mathcal{B}_w^{\infty}$ and $H_{v}^0 = \mathcal{B}_w^0$.

Furthermore, the \emph{Hardy space} $H^p$ for $1 \leq  p < \infty$ consists of all functions $f$ analytic in the unit disc such that
\begin{equation*}
\|f\|_{H^p}^{p} := \sup_{0 \leq r < 1}\tfrac{1}{2\pi}\!\!\hspace*{0.1mm}\int_{0}^{2\pi}{|f(re^{i\theta})|^p d\theta} < \infty,
\end{equation*}
and the \emph{weighted Bergman spaces} for constants $\alpha > -1$ and $1 \leq p < \infty$ are given by
\begin{equation*}
A^p_\alpha = \left\{f\in H(\D): \norm{f}^p_{A^p_\alpha}  := \int_{\D} |f(z)|^p(1-|z|^2)^\alpha \ dA(z) < \infty\right\},
\end{equation*}
where $dA(z)$ is the normalized area measure on $\mathbb{D}$. For further use, recall the functional norms $\|\delta_z\|_{H^p \to \mathbb{C}} = (1 - |z|^2)^{-\frac{1}{p}}$ and $\|\delta_z\|_{A_\alpha^p \to \mathbb{C}} = (1 - |z|^2)^{\frac{-2-\alpha}{p}}$. Finally, the \emph{disc algebra} $A(\mathbb{D})$ is the space of functions analytic on $\mathbb{D}$ that extend continuously to the boundary $\partial\mathbb{D}$.

Condition (I) holds for all the above mentioned spaces, except for $H_v^0$, $\mathcal{B}_v^0$ and $A(\mathbb{D})$. The spaces $H^p$ and $A_{\alpha}^p$ satisfy all conditions (I)-(V) when $1 \leq p < \infty$ and $\alpha > -1$, and the same is true for $H_{v}^{\infty}$ if the weight $v$ is normal and equivalent to its associated weight $\widetilde{v}(z) = \|\delta_z\|_{H_v^{\infty} \to \mathbb{C}}^{-1}$, see \cite{20}. We refer the reader to \cite{5} for a more thorough discussion on the conditions (I)-(V). 

The \emph{essential norm}  of a bounded linear operator $T\colon \mathcal{X}\to\mathcal{Y}$ is defined to be the distance to the compact operators, that is
	\begin{equation*}
	\|T\|_{e,\mathcal{X}\to\mathcal{Y}}= \inf\left\{\|T-K\|_{\mathcal{X}\to\mathcal{Y}}: K\colon \mathcal{X}\to\mathcal{Y} \ \text{is compact}\right\}.
	\end{equation*}
Notice that  $T\colon \mathcal{X}\to\mathcal{Y}$ is compact if and only if $\|T\|_{e,\mathcal{X}\to\mathcal{Y}}=0$. We will use the following result when characterizing compactness and weak compactness of the generalized Volterra operator.

\begin{lem}\label{2}
\cite[Lemma 3.3]{2} Let $\mathcal{X}\subset H(\mathbb{D})$ be a Banach space satisfying condition \textnormal{(I)} and $\mathcal{Y}\subset H(\mathbb{D})$ be a Banach space such that point evaluation functionals on $\mathcal{Y}$ are bounded. Assume that $T\colon \mathcal{X} \to \mathcal{Y}$ is a co-co-continuous linear operator. Then $T\colon \mathcal{X} \to \mathcal{Y}$ is compact (respectively weakly compact) if and only if $\{Tf_n\}_{n=1}^{\infty}$ converges to zero in the norm (respectively in the weak topology) of $\mathcal{Y}$  for each bounded sequence $\{f_n\}_{n=1}^{\infty}$  in $\mathcal{X}$  such that $f_n \to 0$ uniformly on compact subsets of $\mathbb{D}$.
\end{lem}

The above lemma can be applied to the generalized Volterra operator $T_g^{\varphi} \colon \mathcal{X} \to \mathcal{Y}$, since it is obvious that $T_g^{\varphi} \colon H(\mathbb{D}) \to H(\mathbb{D})$ always is co-co-continuous. It is also worth mentioning that if $\mathcal{X},\mathcal{Y} \subset H(\mathbb{D})$ are Banach spaces containing the constant functions and $T_g^{\varphi} \colon \mathcal{X} \to \mathcal{Y}$ is bounded, then  $g \circ \varphi = T_{g}1 + g(0) \in \mathcal{Y}$.

\section{\bf{Boundedness and compactness results for $T_g \colon H_{v_{\alpha}}^{\infty} \to H^{\infty}$}}

In this section we characterize boundedness and compactness of the classical Volterra operator $T_g \colon H_{v_{\alpha}}^{\infty} \to H^{\infty}$ induced by a univalent function $g$ for standard weights $v_{\alpha}$ with $0 \leq \alpha < 1$. The study of the case when $\alpha = 0$, that is when $T_g \colon H^{\infty} \to H^{\infty}$, was initiated in the paper \cite{1} where the authors conjectured that the set
\begin{equation*}
T[H^{\infty}] = \left\{g \in H(\mathbb{D}) : T_g \colon H^{\infty} \to H^{\infty}  \textnormal{ is bounded}\right\}
\end{equation*}
would coincide with the space of functions analytic in $\mathbb{D}$ with bounded radial variation
\begin{equation*}
\textnormal{BRV} = \left\{f \in H(\mathbb{D}) : \sup_{0 \leq \theta < 2\pi}\int_{0}^{1}{|f^{\prime}(re^{i\theta})|\,dr} < \infty \right\}.
\end{equation*}
In \cite{13} this conjecture was confirmed when the inducing function $g$ is univalent, that is 
\begin{equation}\label{36}
T[H^{\infty}] \cap \{g \in H(\mathbb{D}) : g \textnormal{ is univalent}\} = \textnormal{BRV}.
\end{equation}
However, the same paper also contains a counterexample to the general conjecture posed in \cite{1}, meaning that $\textnormal{BRV} \subsetneqq T[H^{\infty}]$. In another recent paper \cite[Section 2.2]{2},  Contreras, Pel{\'a}ez, Pommerenke and R{\"a}tty{\"a}  showed as a side result that $T_g \colon H_{v_1}^{\infty} \to H^{\infty}$ is bounded precisely when $g$ is a constant function. Since the size of the spaces 
$H_{v_{\alpha}}^{\infty}$ increases as the power $\alpha$ grows, one concludes that the only Volterra operator $T_g \colon H_{v_{\alpha}}^{\infty} \to H^{\infty}$ that can be bounded when $\alpha \geq 1$ is the zero operator. Hence, we are left to consider the remaining cases $0 \leq \alpha < 1$, and will also restrict our study to univalent symbols $g$. 

In the mentioned paper \cite{1} the authors also discussed the compactness of the Volterra operator $T_g \colon H^{\infty} \to  H^{\infty}$, and suggested the space
\begin{equation}\label{41}
\textnormal{BRV}_0 = \left\{f \in H(\mathbb{D}) :  \lim_{t \to 1^{-}}\sup_{0 \leq \theta < 2\pi}\int_{t}^{1}{|f^{\prime}(re^{i\theta})|\,dr} = 0\right\}
\end{equation}
of functions analytic in the unit disc with derivative uniformly integrable on radii as a natural candidate for the set of such functions $g$. The notation $\textnormal{BRV}_0$ is adopted from \cite[Section 2.4]{2}. As the main result of this paper, we characterize the compactness of $T_g \colon H_{v_{\alpha}}^{\infty} \to H^{\infty}$ when $0 \leq \alpha < 1$ and $g$ is univalent, proving that the proposed candidate (\ref{41}) is the correct one in the univalent case, see Theorem \ref{38} of this section.

For positive constants $\beta$ and $r$, let $\mathcal{B}\big(\Omega_{\beta}^r\big)$ denote the class of all functions $F$, analytic in the open sector
\begin{equation*}
\Omega_{\beta}^{r} := \left\{z \in \mathbb{C} : 0< |z| < r \textnormal{ and} -\tfrac{\beta}{2} < \textnormal{arg}(z) < \tfrac{\beta}{2}\right\},
\end{equation*}
such that
\begin{equation}\label{39}
\left|F^{\prime}(z)\right| \leq \frac{C_F}{|z|} \ \ \textnormal{for } z \in \Omega_{\beta}^{r}.
\end{equation}
Here $C_F$ is a constant only depending on the function $F$. The main tool used by Smith, Stolyarov and Volberg in \cite{13} when proving (\ref{36}) was the following result concerning uniform approximation of Bloch functions (see also \cite{15}). In the theorem below $\Omega_{\beta} := \Omega_{\beta}^1$ and $\widetilde{u}$ denotes the harmonic conjugate of $u$ with $\widetilde{u}\!\left(\tfrac{1}{2}\right) = 0$.
\begin{thm}\cite[Theorem 1.2]{13}\label{37}
Let $0 < \gamma < \beta < \pi$ and $\varepsilon > 0$. Then there is a number $\delta(\varepsilon) > 0$ such that for each $F \in \mathcal{B}\big(\Omega_{\gamma}^{1/2}\big)$ there exists a harmonic function $u \colon \Omega_{\beta} \to \mathbb{R}$ with the properties
\begin{itemize}
\item[(1)] $|\textnormal{Re}\left(F(x)\right) - u(x)| \leq \varepsilon$, \textnormal{ for} $x \in (0,\delta(\varepsilon)]$.
\item[(2)] $|\widetilde{u}(z)| \leq C(\varepsilon,\gamma,\beta,C_F) < \infty$, \textnormal{ for} $z \in \Omega_{\beta}$. 
\end{itemize}
\end{thm}
\noindent Notice that the number $\delta(\varepsilon)$ is independent of the function $F$, whereas the constant $C(\varepsilon,\gamma,\beta,C_F)$ does depend on $F$ but only through the Bloch constant $C_F$ as defined in (\ref{39}). We are now ready to generalize the result (\ref{36}), which of course appears as the case $\alpha = 0$ in the theorem below.
\begin{thm}\label{32}
If $g \in H(\mathbb{D})$ is univalent and $0 \leq \alpha < 1$, then $T_g \colon H_{v_{\alpha}}^{\infty} \to H^{\infty}$ is bounded if and only if
\begin{equation}\label{29}
\sup_{0 \leq \theta < 2\pi}\int_{0}^{1}{\frac{|g^{\prime}(re^{i\theta})|}{(1-r^2)^{\alpha}}\,dr} < \infty. 
\end{equation}
\end{thm}

\begin{proof}
If $g$ satisfies condition (\ref{29}), then for any $f \in H_{v_{\alpha}}^{\infty}$ we have
\begin{align*}
\|T_g(f)\|_{\infty} &= \sup_{z \in \mathbb{D}}\left|\int_{0}^{z}{f(\xi)g^{\prime}(\xi)d\xi}\right| = \sup_{0 \leq \theta < 2\pi}\sup_{0 \leq R < 1}\left|\int_{0}^{Re^{i\theta}}{f(\xi)g^{\prime}(\xi)d\xi}\right| \\
&= \sup_{0 \leq \theta < 2\pi}\sup_{0 \leq R < 1}\left|\int_{0}^{R}{f(re^{i\theta})g^{\prime}(re^{i\theta})e^{i\theta}dr}\right| \\
&\leq \sup_{0 \leq \theta < 2\pi}\int_{0}^{1}{|f(re^{i\theta})||g^{\prime}(re^{i\theta})|dr} \leq \|f\|_{H_{v_{\alpha}}^{\infty}}\sup_{0 \leq \theta < 2\pi}\int_{0}^{1}{\frac{|g^{\prime}(re^{i\theta})|}{(1-r^2)^{\alpha}}\,dr},
\end{align*}
showing that $T_g \colon H_{v_{\alpha}}^{\infty} \to H^{\infty}$ is bounded.

Assume on the other hand that 
\begin{equation*}
\sup_{0 \leq \theta < 2\pi}\int_{0}^{1}{\frac{|g^{\prime}(re^{i\theta})|}{(1-r^2)^{\alpha}}\,dr} = \infty
\end{equation*}
and choose $n \in \mathbb{N}$. Then there is an angle $0 \leq \theta_n < 2\pi$ such that
\begin{equation}\label{33}
\int_{0}^{1}{\frac{|g^{\prime}(re^{i\theta_n})|}{(1-r^2)^{\alpha}}\,dr} \geq n.
\end{equation}
Following the proof of \cite[Theorem 1.1]{13}  we choose constants $0 < \gamma < \beta < \pi$ and let $\psi_{\beta} \colon \Omega_{\beta} \to \mathbb{D}$ be the conformal map with $\psi_{\beta}(\tfrac{1}{2}) = 0$ and $\psi_{\beta}(0) = 1$. Then as noted in the mentioned proof, there is a constant $C(\gamma,\beta)$, only depending on $\gamma$ and $\beta$, such that
\begin{equation}\label{40}
|z\psi_{\beta}^{\prime}(z)| \leq C(\gamma,\beta)\left(1- |\psi_{\beta}(z)|^2\right)
\end{equation}
for every $z \in \Omega_{\gamma}^{1/2}$. We need to consider a sector rotated by the angle $\theta_n$, and hence introduce
\begin{equation*}
\Omega_{\beta,n}^r := \left\{z \in \mathbb{C} : 0< |z| < r \textnormal{ and } \theta_n - \tfrac{\beta}{2} < \textnormal{arg}(z) < \theta_n + \tfrac{\beta}{2}\right\}
\end{equation*}
and denote $\Omega_{\beta,n} := \Omega_{\beta,n}^1$. By defining $\psi_{\beta,n}(z):= e^{i\theta_n}\psi_{\beta}(e^{-i\theta_n}z)$ we obtain a conformal map $\psi_{\beta,n} \colon \Omega_{\beta,n} \to \mathbb{D}$ such that $\psi_{\beta,n}(\frac{1}{2}e^{i\theta_n}) = 0$ and $\psi_{\beta,n}(0) = e^{i\theta_n}$. Let the function $G_n \colon \Omega_{\beta,n} \to \mathbb{C}$ be given by 
\begin{equation*}
G_n(z) := -i\log\left(g^{\prime} \circ \psi_{\beta,n}(z)\right),
\end{equation*}
and define $F_n \colon \Omega_{\beta} \to \mathbb{C}$ as $F_n(z) := G_n(e^{i\theta_n}z)$. Since $g$ is univalent it holds that
\begin{equation*}
(1-|z|^2)\frac{|g^{\prime\prime}(z)|}{|g^{\prime}(z)|} \leq 6
\end{equation*}
for every $z \in \mathbb{D}$, see \cite[p. 9]{16}. This gives
\begin{align*}
|F_n^{\prime}(z)| = \frac{|g^{\prime\prime}(\psi_{\beta,n}(e^{i\theta_n}z))|}{|g^{\prime}(\psi_{\beta,n}(e^{i\theta_n}z))|}|\psi_{\beta,n}^{\prime}(e^{i\theta_n}z)| \leq 6\frac{|\psi_{\beta,n}^{\prime}(e^{i\theta_n}z)|}{1-|\psi_{\beta,n}(e^{i\theta_n}z)|^2} = 6\frac{|\psi_{\beta}^{\prime}(z)|}{1-|\psi_{\beta}(z)|^2} \leq 6\frac{C(\gamma,\beta)}{|z|}
\end{align*}
for every $z \in \Omega_{\gamma}^{1/2}$ by (\ref{40}), and hence the restriction of $F_n$ to $\Omega_{\gamma}^{1/2}$ belongs to $\mathcal{B}\big(\Omega_{\gamma}^{1/2}\big)$. Using Theorem \ref{37} with $\varepsilon = \frac{\pi}{4}$ provides us with a constant $\delta\!\left(\frac{\pi}{4}\right)$ and a harmonic function $U_n \colon \Omega_{\beta} \to \mathbb{R}$ with a corresponding harmonic conjugate $\widetilde{U}_n \colon \Omega_{\beta} \to \mathbb{R}$ such that
\begin{equation*}
\left|\textnormal{Re}\left(F_n(x)\right) - U_n(x)\right| \leq \tfrac{\pi}{4} \ \ \textnormal{for } x \in (0,\delta\!\left(\tfrac{\pi}{4}\right)]
\end{equation*}
and 
\begin{equation*}
\big|\widetilde{U}_n(z)\big| \leq C(\gamma,\beta,C_{F_n}) = C_1(\gamma,\beta) \ \ \textnormal{for } z \in \Omega_{\beta}.
\end{equation*}
The constant $C_1(\gamma,\beta)$ is independent of $n$ because the Bloch constant $C_{F_n} = 6\hspace*{0.2mm}C(\gamma,\beta)$  only depends on $\gamma$ and $\beta$. Rotating back to $\Omega_{\beta,n}$ by defining $u_n(z) := U_n(e^{-i\theta_n}z)$ we obtain a harmonic function $u_n \colon \Omega_{\beta,n} \to \mathbb{R}$ with harmonic conjugate $\widetilde{u}_n(z) := \widetilde{U}_n(e^{-i\theta_n}z) \colon \Omega_{\beta,n} \to \mathbb{R}$ such that
\begin{equation}\label{35}
\big|\textnormal{Re}\big(G_n(xe^{i\theta_n})\big) - u_n(xe^{i\theta_n})\big| \leq \tfrac{\pi}{4} \ \ \textnormal{for } x \in (0,\delta\!\left(\tfrac{\pi}{4}\right)]
\end{equation}
and 
\begin{equation*}
|\widetilde{u}_n(z)| \leq C_1(\gamma,\beta) \ \ \textnormal{for } z \in \Omega_{\beta,n}.
\end{equation*}
Moreover, from (\ref{35}) it follows that there is a constant $0 < r_{\beta} < 1$, also independent of $n$ by a rotational argument via $\Omega_{\beta}$, such that
\begin{equation*}
\big|\text{arg}(g^{\prime}(re^{i\theta_n})) - u_n(\psi_{\beta,n}^{-1}(re^{i\theta_n}))\big| \leq \tfrac{\pi}{4}  \ \ \textnormal{for } r_{\beta} \leq r < 1.
\end{equation*}
Now define $h_n(z) := e^{-i(u_n(z) + i\hspace*{0.2mm}\widetilde{u}_n(z))}$ and consider the sequence $\{f_n\}_{n=1}^{\infty}$ of test functions given by
\begin{equation}\label{34}
f_n(z) := \dfrac{h_n \circ \psi_{\beta,n}^{-1}(z)}{(1-e^{-2i\theta_n}z^2)^{\alpha}}.
\end{equation}
Each $f_n$ belongs to $H_{v_{\alpha}}^{\infty}$, with uniformly bounded norm:
\begin{equation*}
\|f_n\|_{H_{v_{\alpha}}^{\infty}} = \sup_{z \in \mathbb{D}}{\left(1-|z|^2\right)^{\alpha}\left|\dfrac{h_n \circ \psi_{\beta,n}^{-1}(z)}{(1-e^{-2i\theta_n}z^2)^{\alpha}}\right|} \leq \|h_n \circ \psi_{\beta,n}^{-1}\|_{\infty} \leq e^{C_1(\gamma,\beta)}.
\end{equation*}
In order to contradict the boundedness of $T_g \colon H_{v_{\alpha}}^{\infty} \to H^{\infty}$, let $r_{\beta} < t < 1$ and observe that
\begin{align*}
\|T_g(f_n)\|_{\infty} &= \sup_{z \in \mathbb{D}}\left|\int_{0}^{z}{f_n(\xi)g^{\prime}(\xi)d\xi}\right| \geq \left|\int_{0}^{te^{i\theta_n}}{f_n(\xi)g^{\prime}(\xi)d\xi}\right| = \left|\int_{0}^{t}{f_n(re^{i\theta_n})g^{\prime}(re^{i\theta_n})e^{i\theta_n}dr}\right| \\
&\geq \textnormal{Re}\left(\int_{0}^{r_{\beta}}{f_n(re^{i\theta_n})g^{\prime}(re^{i\theta_n})dr}\right) + \textnormal{Re}\left(\int_{r_{\beta}}^{t}{f_n(re^{i\theta_n})g^{\prime}(re^{i\theta_n})dr}\right).
\end{align*}
The first term in the last expression can be estimated as follows
\begin{align*}
\left|\textnormal{Re}\left(\int_{0}^{r_{\beta}}{f_n(re^{i\theta_n})g^{\prime}(re^{i\theta_n})dr}\right)\right| &\leq \int_{0}^{r_{\beta}}{|f_n(re^{i\theta_n})||g^{\prime}(re^{i\theta_n})|dr} \\
&\leq \|f_n\|_{H_{v_{\alpha}}^{\infty}}\int_{0}^{r_{\beta}}{\frac{|g^{\prime}(re^{i\theta_n})|}{(1-r^2)^{\alpha}}dr} \\
&\leq \frac{e^{C_1(\gamma,\beta)}C_2(\beta)}{(1-r_{\beta}^2)^{\alpha}},
\end{align*}
where $C_2(\beta) := \sup\{|g^{\prime}(z)| : |z| \leq r_{\beta}\} > 0$, and for the second term we have
\begin{align*}
&\ \ \ \,\hspace*{0.5mm}\textnormal{Re}\left(\int_{r_{\beta}}^{t}{f_n(re^{i\theta_n})g^{\prime}(re^{i\theta_n})dr}\right) \\
&= \textnormal{Re}\left(\int_{r_{\beta}}^{t}{\frac{g^{\prime}(re^{i\theta_n})}{(1-r^2)^{\alpha}}e^{-i\left(u_n\left(\psi_{\beta,n}^{-1}(re^{i\theta_n})\right) + i\hspace*{0.2mm}\widetilde{u}_n\left(\psi_{\beta,n}^{-1}(re^{i\theta_n})\right)\right)}dr}\right)  \\
&= \int_{r_{\beta}}^{t}{\frac{|g^{\prime}(re^{i\theta_n})|}{(1-r^2)^{\alpha}}e^{\widetilde{u}_n\left(\psi_{\beta,n}^{-1}(re^{i\theta_n})\right)}\textnormal{Re}\left(e^{i\left(\textnormal{arg}(g^{\prime}(re^{i\theta_n})) - u_n\left(\psi_{\beta,n}^{-1}(re^{i\theta_n})  \right)\right)}\right)dr} \\
&\geq \cos\left(\tfrac{\pi}{4}\right)e^{-C_1(\gamma,\beta)}\int_{r_{\beta}}^{t}{\frac{|g^{\prime}(re^{i\theta_n})|}{(1-r^2)^{\alpha}}dr} \\
&\geq \cos\left(\tfrac{\pi}{4}\right)e^{-C_1(\gamma,\beta)}\left(\int_{0}^{t}{\frac{|g^{\prime}(re^{i\theta_n})|}{(1-r^2)^{\alpha}}dr} - \frac{C_2(\beta)}{(1-r_{\beta}^2)^{\alpha}}\right).
\end{align*}
Hence, for every $r_{\beta} < t < 1$ it holds that
\begin{equation*}
\|T_g(f_n)\|_{\infty} \geq \cos\left(\tfrac{\pi}{4}\right)e^{-C_1(\gamma,\beta)}\left(\int_{0}^{t}{\frac{|g^{\prime}(re^{i\theta_n})|}{(1-r^2)^{\alpha}}dr} - \frac{C_2(\beta)}{(1-r_{\beta}^2)^{\alpha}}\right) - \frac{e^{C_1(\gamma,\beta)}C_2(\beta)}{(1-r_{\beta}^2)^{\alpha}}.
\end{equation*}
Letting $t \to 1$ and using (\ref{33}) we arrive at the estimate
\begin{equation*}
\|T_g(f_n)\|_{\infty} \geq \cos\left(\tfrac{\pi}{4}\right)e^{-C_1(\gamma,\beta)}\left(n - \frac{C_2(\beta)}{(1-r_{\beta}^2)^{\alpha}}\right) - \frac{e^{C_1(\gamma,\beta)}C_2(\beta)}{(1-r_{\beta}^2)^{\alpha}},
\end{equation*}
showing that $T_g \colon H_{v_{\alpha}}^{\infty} \to H^{\infty}$ is unbounded.
\end{proof}

We now come to the main result of the paper, which  in the univalent case answers Problem 4.4 posed in \cite{1} concerning the compactness of $T_g \colon H^{\infty} \to H^{\infty}$, and also the compactness of $T_g \colon A(\mathbb{D}) \to A(\mathbb{D})$ due to \cite[Theorem 1.7]{2}.
\begin{thm}\label{38}
If $g \in H(\mathbb{D})$ is univalent and $0 \leq \alpha < 1$, then $T_g \colon H_{v_{\alpha}}^{\infty} \to H^{\infty}$ is compact if and only if
\begin{equation}\label{30}
\lim_{t \to 1^{-}}\sup_{0 \leq \theta < 2\pi}\int_{t}^{1}{\frac{|g^{\prime}(re^{i\theta})|}{(1-r^2)^{\alpha}}\,dr} = 0. 
\end{equation}
\end{thm}

\begin{proof}
Assume first that $g$ satisfies condition (\ref{30}). To prove that $T_g \colon H_{v_{\alpha}}^{\infty} \to H^{\infty}$ is compact, choose an arbitrary sequence $\{f_n\}_{n=1}^{\infty}\subset H_{v_{\alpha}}^{\infty}$ such that $\sup_{n \in \mathbb{N}}\|f_n\|_{H_{v_{\alpha}}^{\infty}} < \infty$ and $f_n \xrightarrow{co} 0$, and let $\varepsilon > 0$. Then there is $0 < t_{\varepsilon} < 1$ such that 
\begin{equation*}
\sup_{0 \leq \theta < 2\pi}\int_{t_{\varepsilon}}^{1}{\frac{|g^{\prime}(re^{i\theta})|}{(1-r^2)^{\alpha}}\,dr} < \frac{\varepsilon}{1+\sup_{k \in \mathbb{N}}\|f_k\|_{H_{v_{\alpha}}^{\infty}}}.
\end{equation*}
Denote $M_{\varepsilon} := \sup_{|z| \leq t_{\varepsilon}}|g^{\prime}(z)| > 0$ and choose an integer $N_{\varepsilon}$ such that 
\begin{equation*}
\sup_{|z| \leq t_{\varepsilon}}|f_n(z)| < \frac{\varepsilon}{(1+\sup_{k \in \mathbb{N}}\|f_k\|_{H_{v_{\alpha}}^{\infty}})M_{\varepsilon}}
\end{equation*}
whenever $n > N_{\varepsilon}$. Now if $n > N_{\varepsilon}$, then
\begin{align*}
\|T_g(f_n)\|_{\infty} &\leq \sup_{0 \leq \theta < 2\pi}\int_{0}^{1}{|f_n(re^{i\theta})||g^{\prime}(re^{i\theta})|dr} \\
&\leq \sup_{0 \leq \theta < 2\pi}\int_{0}^{t_{\varepsilon}}{|f_n(re^{i\theta})||g^{\prime}(re^{i\theta})|dr} + \sup_{0 \leq \theta < 2\pi}\int_{t_{\varepsilon}}^{1}{|f_n(re^{i\theta})||g^{\prime}(re^{i\theta})|dr} \\
&\leq M_{\varepsilon}t_{\varepsilon}\sup_{|z| \leq t_{\varepsilon}}|f_n(z)| + \sup_{k \in \mathbb{N}}\|f_k\|_{H_{v_{\alpha}}^{\infty}}\sup_{0 \leq \theta < 2\pi}\int_{t_{\varepsilon}}^{1}{\frac{|g^{\prime}(re^{i\theta})|}{(1-r^2)^{\alpha}}\,dr} < \varepsilon,
\end{align*}
which shows that $\lim_{n \to \infty}\|T_g(f_n)\|_{\infty} = 0$ and $T_g \colon H_{v_{\alpha}}^{\infty} \to H^{\infty}$ is compact by Lemma \ref{2}.

On the other hand, if (\ref{30}) does not hold then 
\begin{equation*}
c := \limsup_{t \to 1^{-}}\sup_{0 \leq \theta < 2\pi}\int_{t}^{1}{\frac{|g^{\prime}(re^{i\theta})|}{(1-r^2)^{\alpha}}\,dr} > 0.
\end{equation*}
We may assume that the supremum in (\ref{29}) is finite, because otherwise $T_g \colon H_{v_{\alpha}}^{\infty} \to H^{\infty}$ would not be bounded and hence not compact. This ensures that $c$ is a finite constant. Let $\{t_n\}_{n=1}^{\infty}$ be a sequence such that $0 < t_n < t_{n+1} < 1$, $\lim_{n \to \infty}t_n = 1$ and
\begin{equation*}
\inf_{n \in \mathbb{N}}\sup_{0 \leq \theta < 2\pi}\int_{t_n}^{1}{\frac{|g^{\prime}(re^{i\theta})|}{(1-r^2)^{\alpha}}\,dr} = \lim_{n \to \infty}\sup_{0 \leq \theta < 2\pi}\int_{t_n}^{1}{\frac{|g^{\prime}(re^{i\theta})|}{(1-r^2)^{\alpha}}\,dr} = c.
\end{equation*}
For every $n \in \mathbb{N}$ one can then choose an angle $0 \leq \theta_n < 2\pi$ such that
\begin{equation*}
\int_{t_n}^{1}{\frac{|g^{\prime}(re^{i\theta_n})|}{(1-r^2)^{\alpha}}\,dr} > c -\tfrac{1}{n}.
\end{equation*}
Furthermore, since $\lim_{n \to \infty}\frac{\log \frac{1}{2}}{\log t_n} = +\infty$, we may also assume that $\frac{\log \frac{1}{2}}{\log t_n} \geq 1$ for every $n \in \mathbb{N}$ and define a sequence of positive integers $\{k_n\}_{n=1}^{\infty}$ as $k_n := \Big\lfloor\frac{\log \frac{1}{2}}{\log t_n}\Big\rfloor$. Then $t_n^{k_n} \geq \frac{1}{2}$ for every $n \in \mathbb{N}$ and $\lim_{n \to \infty}k_n = +\infty$. Modifying the test functions used in the proof of Theorem \ref{32} as follows
\begin{equation*}
s_n(z) := f_n(z)z^{k_n} = \dfrac{h_n \circ \psi_{\beta,n}^{-1}(z)}{(1-e^{-2i\theta_n}z^2)^{\alpha}}z^{k_n},
\end{equation*}
where $f_n$ is given by (\ref{34}), we obtain a sequence $\{s_n\}_{n=1}^{\infty} \subset H_{v_{\alpha}}^{\infty}$ such that $s_n \xrightarrow{co} 0$ and $\sup_{n \in \mathbb{N}}\|s_n\|_{H_{v_{\alpha}}^{\infty}} \leq e^{C_1(\gamma,\beta)}$. To prove that $T_g \colon H_{v_{\alpha}}^{\infty} \to H^{\infty}$ is non-compact we estimate
\begin{align*}
\|T_g(s_n)\|_{\infty} &\geq \limsup_{t \to 1^{-}}\textnormal{Re}\left(\int_{0}^{t}{f_n(re^{i\theta_n})g^{\prime}(re^{i\theta_n})r^{k_n}dr}\right) \\
&= \textnormal{Re}\left(\int_{0}^{1}{f_n(re^{i\theta_n})g^{\prime}(re^{i\theta_n})r^{k_n}dr}\right),
\end{align*}
where the last integral is convergent because the supremum in (\ref{29}) is finite by assumption. Moreover, since
\begin{align*}
\left|\textnormal{Re}\left(\int_{0}^{r_{\beta}}{f_n(re^{i\theta_n})g^{\prime}(re^{i\theta_n})r^{k_n}dr}\right)\right| &\leq \int_{0}^{r_{\beta}}{|f_n(re^{i\theta_n})||g^{\prime}(re^{i\theta_N})|r^{k_n}dr} \\
&\leq \|f_n\|_{H_{v_{\alpha}}^{\infty}}\int_{0}^{r_{\beta}}{\frac{|g^{\prime}(re^{i\theta_N})|}{(1-r^2)^{\alpha}}r^{k_n}dr} \\
&\leq \frac{e^{C_1(\gamma,\beta)}C_2(\beta)}{(1-r_{\beta}^2)^{\alpha}}r_{\beta}^{k_n},
\end{align*}
and for $n$ large enough (there is some integer $N_{\beta}$ such that $t_n > r_{\beta}$ whenever $n \geq N_{\beta}$)
\begin{align*}
&\ \ \ \,\hspace*{0.5mm}\textnormal{Re}\left(\int_{r_{\beta}}^{1}{f_n(re^{i\theta_n})g^{\prime}(re^{i\theta_n})r^{k_n}dr}\right) \\
&= \int_{r_{\beta}}^{1}{\frac{|g^{\prime}(re^{i\theta_n})|}{(1-r^2)^{\alpha}}e^{\widetilde{u}_n\left(\psi_{\beta,n}^{-1}(re^{i\theta_n})\right)}\textnormal{Re}\left(e^{i\left(\textnormal{arg}(g^{\prime}(re^{i\theta_n})) - u_n\left(\psi_{\beta,n}^{-1}(re^{i\theta_n})  \right)\right)}\right)r^{k_n}dr} \\
&\geq \cos\left(\tfrac{\pi}{4}\right)e^{-C_1(\gamma,\beta)}\int_{r_{\beta}}^{1}{\frac{|g^{\prime}(re^{i\theta_n})|}{(1-r^2)^{\alpha}}r^{k_n}dr}  \geq  \cos\left(\tfrac{\pi}{4}\right)e^{-C_1(\gamma,\beta)}t_n^{k_n}\int_{t_n}^{1}{\frac{|g^{\prime}(re^{i\theta_n})|}{(1-r^2)^{\alpha}}dr} \\
&\geq \tfrac{1}{2}\cos\left(\tfrac{\pi}{4}\right)e^{-C_1(\gamma,\beta)}\left(c - \tfrac{1}{n}\right),
\end{align*}
we see that 
\begin{align*}
\limsup_{n \to \infty}\|T_g(s_n)\|_{\infty} &\geq \lim_{n \to \infty}\left(\tfrac{1}{2}\cos\left(\tfrac{\pi}{4}\right)e^{-C_1(\gamma,\beta)}\left(c - \tfrac{1}{n}\right) - \frac{e^{C_1(\gamma,\beta)}C_2(\beta)}{(1-r_{\beta}^2)^{\alpha}}r_{\beta}^{k_n}\right) \\
&= \tfrac{1}{2}\cos\left(\tfrac{\pi}{4}\right)e^{-C_1(\gamma,\beta)}c > 0,
\end{align*}
and the proof is complete.
\end{proof}

\section{\bf{Bounded, weakly compact and compact generalized Volterra operators}}
In this section we study boundedness, compactness and weak compactness of the generalized Volterra operator $T_g^{\varphi}$ mapping between various Banach spaces of analytic functions. Among other things, we will use the norm and essential norm formulas of weighted composition operators obtained in \cite{5} to estimate the norm and essential norm of $T_g^{\varphi}$ mapping into $H_v^{\infty}$ and $\mathcal{B}_v^{\infty}$, as well as show that the notions of compactness and weak compactness coincide when $T_g^{\varphi}$ acts on $H_v^{\infty}$ or $\mathcal{B}_v^{\infty}$. We begin with a general operator theoretic result concerning boundedness, which is a generalization of \cite[Proposition 2.1]{2}.

\begin{thm}\label{thm:bdd1}
Let $\mathcal{X}\subset H(\mathbb{D})$ be a Banach space containing the disc algebra and satisfying conditions \textnormal{(I)} and \textnormal{(IV)}. Let $\mathcal{Y}\subset H(\mathbb{D})$ be a Banach space satisfying condition \textnormal{(I)}. If the operator $T\colon H(\mathbb{D}) \to H(\mathbb{D})$ is co-co-continuous, then the following statements are equivalent\textnormal{:}
\begin{itemize}
\item[(i)] $T\colon \mathcal{X} \to \mathcal{Y}$ is bounded,
\item[(ii)] $T\colon \overline{\mathcal{P}}^{\mathcal{X}} \to \mathcal{Y}$ is bounded,
\end{itemize}
and so are the operator norms\textnormal{:} $\|T\|_{\mathcal{X} \to \mathcal{Y}} \asymp \|T\|_{\overline{\mathcal{P}}^{\mathcal{X}} \to \mathcal{Y}}$.
\end{thm}
\begin{proof}
Clearly (i) implies (ii). Conversely, let $T\colon \overline{\mathcal{P}}^{\mathcal{X}} \to \mathcal{Y}$ be bounded, choose $f\in \mathcal{X}$ and let $\{r_n\}_{n=1}^{\infty} \subset (0,1)$ be a sequence such that $r_n \to 1$ as $n \to \infty$. Then $f_{r_n}\in \overline{\mathcal{P}}^{\mathcal{X}}$ since $\mathcal{X}$ contains the disc algebra. Also, we have that $f_{r_n} \xrightarrow{co} f$ and therefore $Tf_{r_n} \xrightarrow{co} Tf$. Note that $\{Tf_{r_n}\}_{n=1}^{\infty}$ is a bounded sequence in $\mathcal{Y}$, since by condition (IV) for $\mathcal{X}$ it holds that
\begin{align*}
\|Tf_{r_n}\|_{\mathcal{Y}} &=\|T(T_{r_n}f)\|_{\mathcal{Y}}\nonumber \leq \|T\|_{\overline{\mathcal{P}}^{\mathcal{X}} \to \mathcal{Y}} \|T_{r_n}\|_{\mathcal{X} \to \mathcal{X}}\|f\|_{\mathcal{X}}\\
&\leq \|T\|_{\overline{\mathcal{P}}^{\mathcal{X}} \to \mathcal{Y}} \sup_{0<r<1}\|T_r\|_{\mathcal{X} \to \mathcal{X}}\|f\|_{\mathcal{X}}.
\end{align*}
This implies that the sequence $\{Tf_{r_n}\}_{n=1}^{\infty}$ belongs to the closed ball $B_{\mathcal{Y}}(0,R)$ with radius $R=\|T\|_{\overline{\mathcal{P}}^{\mathcal{X}} \to \mathcal{Y}} \sup_{0<r<1}\|T_r\|_{\mathcal{X} \to \mathcal{X}}\|f\|_{\mathcal{X}}$.
Thus, since $\mathcal{Y}$ satisfies condition (I), there exist $g \in B_{\mathcal{Y}}(0,R)$ and a subsequence $\{Tf_{r_{n_k}}\}_{k=1}^{\infty}$ such that $Tf_{r_{n_k}}\xrightarrow{co} g$. Therefore $Tf=g\in \mathcal{Y}$, meaning that $T \colon \mathcal{X} \to \mathcal{Y}$ is well-defined and hence bounded by the closed graph theorem. Moreover, since $g \in B_{\mathcal{Y}}(0,R)$, we have
\begin{equation}\label{eq:Tfg}
\|Tf\|_{\mathcal{Y}}=\|g\|_{\mathcal{Y}}\leq R = \|T\|_{\overline{\mathcal{P}}^{\mathcal{X}} \to \mathcal{Y}} \sup_{0<r<1}\|T_r\|_{\mathcal{X}\to \mathcal{X}}\|f\|_{\mathcal{X}},
\end{equation}
which implies that $\|T\|_{\mathcal{X} \to \mathcal{Y}} \lesssim \|T\|_{\overline{\mathcal{P}}^{\mathcal{X}} \to \mathcal{Y}}$ and completes the proof.
\end{proof}

\begin{rem}\label{27}
Theorem \ref{thm:bdd1} can for example be applied to the operator $T_g^{\varphi} \colon \mathcal{X} \to \mathcal{Y}$ in the following special cases:
\begin{itemize}
\item[(i)] $\mathcal{X}=H^\infty$ and $\overline{\mathcal{P}}^{\mathcal{X}} = A(\mathbb{D})$.
\item[(ii)] $\mathcal{X}=H^\infty_v$ and $\overline{\mathcal{P}}^{\mathcal{X}} = H^0_v$.
\item[(iii)] $\mathcal{X}=\mathcal{B}_v^\infty$ and $\overline{\mathcal{P}}^{\mathcal{X}} = \mathcal{B}_v^0$.
\item[(iv)] $\mathcal{X} = \textnormal{BMOA}$ and $\overline{\mathcal{P}}^{\mathcal{X}} = \textnormal{VMOA}$.
\end{itemize}
From the estimate (\ref{eq:Tfg}) it also follows that $\|T\|_{\mathcal{X} \to \mathcal{Y}} = \|T\|_{\overline{\mathcal{P}}^{\mathcal{X}} \to \mathcal{Y}}$ if the space $\mathcal{X}$ satisfies the stronger condition $\sup_{0<r<1}\|T_r\|_{\mathcal{X} \to \mathcal{X}} \leq 1$, which for example is the case when $\mathcal{X}=H^\infty$ or $\mathcal{X}=H^\infty_v$.
\end{rem}

In the next theorem we use results from \cite{5} to estimate the norm of $T_g^{\varphi}$ mapping into the spaces $H_v^{\infty}$ and $\mathcal{B}_v^{\infty}$ in terms of the inducing symbols $\varphi$ and $g$.

\begin{thm}\label{24}
Let $\mathcal{X}$ be a Banach space of analytic functions on $\mathbb{D}$ satisfying condition \textnormal{(I)}.
\begin{itemize}
\item[(i)]  If the weight $v$ is normal, then
\begin{equation*}
\|T_g^{\varphi}\|_{\mathcal{X} \to H_v^{\infty}} \asymp \sup_{z \in \mathbb{D}}(1-|z|)v(z)|(g \circ \varphi)^{\prime}(z)|\|\delta_{\varphi(z)}\|_{\mathcal{X} \to \mathbb{C}}.
\end{equation*}
\end{itemize}
\begin{itemize}
\item[(ii)] For any weight $v$,
\begin{equation*}
\|T_g^{\varphi}\|_{\mathcal{X} \to \mathcal{B}_v^{\infty}} = \sup_{z \in \mathbb{D}}v(z)|(g \circ \varphi)^{\prime}(z)|\|\delta_{\varphi(z)}\|_{\mathcal{X} \to \mathbb{C}}.
\end{equation*}
\end{itemize}
\end{thm}

\begin{proof}
In order to prove (i), let $w(z) := (1- |z|)v(z)$ and consider the operators $D(f)(z) = f^{\prime}(z) $ and $I(f)(z) = \int_{0}^{z}{f(\xi)d\xi}$. The differentiation operator $D \colon H_v^\infty \to H_{w}^\infty$ and integral operator $I \colon H_{w}^\infty \to H_v^\infty$ are well-defined because $v$ is normal, see \cite[Theorem 3.3]{9}, and thus bounded by the closed graph theorem. Furthermore, $I \circ D(h) = h - h(0)$ and $T_g^{\varphi}(f)(0) = 0$ for every $h \in H_v^{\infty}$ and $f \in \mathcal{X}$, so that
\begin{equation*}
T_g^{\varphi} = I \circ D \circ T_g^{\varphi} + \delta_0 \circ T_g^{\varphi} = I \circ D \circ T_g^{\varphi},
\end{equation*}
and hence
\begin{equation*}
\|D\|_{H_v^\infty \to H_{w}^\infty}^{-1}\|D \circ T_g^{\varphi}\|_{\mathcal{X} \to H_w^{\infty}} \leq \|T_g^{\varphi}\|_{\mathcal{X} \to H_v^{\infty}} \leq \|I\|_{H_{w}^\infty \to H_v^\infty}\|D \circ T_g^{\varphi}\|_{\mathcal{X} \to H_w^{\infty}}.
\end{equation*}
Now since $D \circ T_g^{\varphi} = (g \circ \varphi)^{\prime}C_{\varphi}$ is a weighted composition operator, we finally get that
\begin{align*}
\|T_g^{\varphi}\|_{\mathcal{X} \to H_v^{\infty}} &\asymp \|D \circ T_g^{\varphi}\|_{\mathcal{X} \to H_w^{\infty}} =  \| (g \circ \varphi)^{\prime}C_{\varphi}\|_{\mathcal{X} \to H_w^{\infty}} \\
&= \sup_{z \in \mathbb{D}}(1-|z|)v(z)|(g \circ \varphi)^{\prime}(z)|\|\delta_{\varphi(z)}\|_{\mathcal{X} \to \mathbb{C}}
\end{align*}
by \cite[Corollary 3.2]{5}. The same corollary can also be used to prove (ii), as follows:
\begin{equation*}
\|T_g^{\varphi}\|_{\mathcal{X} \to \mathcal{B}_v^{\infty}} = \|D \circ T_g^{\varphi}\|_{\mathcal{X} \to H_v^{\infty}} = \|(g \circ \varphi)^{\prime}C_{\varphi}\|_{\mathcal{X} \to H_v^{\infty}} =  \sup_{z \in \mathbb{D}}v(z)|(g \circ \varphi)^{\prime}(z)|\|\delta_{\varphi(z)}\|_{\mathcal{X} \to \mathbb{C}},
\end{equation*}
and we are done.
\end{proof}

We now turn to investigate weak compactness and compactness of the generalized Volterra operator. The following lemma ensures the existence of a predual operator under very general assumptions.

\begin{lem}\label{12}
Let $\mathcal{X}, \mathcal{Y} \subset H(\mathbb{D})$ be Banach spaces satisfying condition \textnormal{(I)}. If the operator $T \colon \mathcal{X} \to \mathcal{Y}$ is bounded and the restriction $T|B_{\mathcal{X}}$ is co-co-continuous, then the operator
\begin{equation*}
\Phi_{\mathcal{Y}} \circ T \circ \Phi_{\mathcal{X}}^{-1} \colon (^*\!{\mathcal{X}})^* \to (^*{\mathcal{Y}})^*
\end{equation*}
is $\sigma((^*\!{\mathcal{X}})^*,{^*\!\mathcal{X}})-\sigma((^*{\mathcal{Y}})^*,{^*\mathcal{Y}})$ continuous, and  consequently there exists a bounded operator $S \colon {^*\mathcal{Y}} \to {^*\!\mathcal{X}}$ such that $T = \Phi_{\mathcal{Y}}^{-1}\circ S^* \circ \Phi_{\mathcal{X}}$.
\end{lem}

\begin{proof}
Choose an arbitrary net $\{\ell_{\gamma} = \Phi_{\mathcal{X}}(f_{\gamma})\} \subset (^*\!{\mathcal{X}})^*$ such that $\ell_{\gamma} \xrightarrow{w^{\ast}} 0$, where $f_{\gamma} \in \mathcal{X}$, and let $u\in {^*\mathcal{Y}}$. We have that $u\circ T \in{^*\!\mathcal{X}}$, since if $\{h_n\}_{n=1}^{\infty} \subset B_{\mathcal{X}}$ is such that $h_n \xrightarrow{co} 0$ then $T(h_n) \xrightarrow{co} 0$ by assumption, and hence
\begin{equation*}
\lim_{n \to \infty}u(T(h_n)) = \|T\|_{\mathcal{X} \to \mathcal{Y}}\lim_{n \to \infty}u\left(\frac{T(h_n)}{\|T\|_{\mathcal{X} \to \mathcal{Y}}}\right) = 0
\end{equation*}
because $u|B_{\mathcal{Y}}$ is $co$-continuous and
\begin{equation*}
\left\|\frac{T(h_n)}{\|T\|_{\mathcal{X} \to \mathcal{Y}}}\right\|_{\mathcal{Y}} \leq \|h_n\|_{\mathcal{X}} \leq 1.
\end{equation*}
Now from
\begin{align*}
\big((\Phi_{\mathcal{Y}} \circ T \circ \Phi_{\mathcal{X}}^{-1})(\ell_{\gamma})\big)(u) &=  \big((\Phi_{\mathcal{Y}} \circ T \circ \Phi_{\mathcal{X}}^{-1})(\Phi_{\mathcal{X}}(f_{\gamma}))\big)(u) \\
&= \big(\Phi_{\mathcal{Y}} \circ T(f_{\gamma})\big)(u) = u(T(f_{\gamma})) \\ 
&= \big(\Phi_{\mathcal{X}}(f_{\gamma})\big)(u\circ T) = \ell_{\gamma}(u\circ T)
\end{align*}
it follows that
\begin{equation*}
(\Phi_{\mathcal{Y}} \circ T \circ \Phi_{\mathcal{X}}^{-1})(\ell_{\gamma}) \xrightarrow{w^{\ast}} 0,
\end{equation*}
and hence $\Phi_{\mathcal{Y}} \circ T \circ \Phi_{\mathcal{X}}^{-1} \colon (^*{\mathcal{X}})^* \to (^*{\mathcal{Y}})^*$ is $w^{\ast}$-$w^{\ast}$-continuous. The last statement of the lemma follows from \cite[Theorem 3.1.11]{8}. 
\end{proof}

Using this predual operator $S$ one can show that compactness and weak compactness coincide for a large class of operators mapping from $H^\infty_v$ or $\mathcal B^\infty_v$ into a Banach space $\mathcal Y\subset H(\mathbb D)$.

\begin{thm}\label{13}
Let $v$ be a normal weight and assume that the Banach space $\mathcal{Y} \subset H(\mathbb{D})$ satisfies condition \textnormal{(I)}.
\begin{itemize}
\item[(i)] If the restriction $T|B_{H_{v}^{\infty}}$ is co-co-continuous then $T \colon H_{v}^{\infty} \to \mathcal{Y}$ is compact if and only if it is weakly compact.
\item[(ii)] If the restriction $T|B_{\mathcal{B}_{v}^{\infty}}$ is co-co-continuous then $T \colon \mathcal{B}_{v}^{\infty} \to \mathcal{Y}$ is compact if and only if it is weakly compact.
\end{itemize}

\end{thm}
\begin{proof}
To prove (i), suppose $T \colon H_{v}^{\infty} \to \mathcal{Y}$ is weakly compact. By Lemma \ref{12} there exists a bounded operator $S \colon {^*\mathcal{Y}} \to {^*\!\hspace*{0.2mm}H_{v}^{\infty}}$ such that $T = \Phi_{\mathcal{Y}}^{-1}\circ S^* \circ \Phi_{H_{v}^{\infty}}$, and hence by the Gantmacher theorem $S \colon {^*\mathcal{Y}} \to {^*\!\hspace*{0.2mm}H_{v}^{\infty}}$ is weakly compact. Moreover, ${^*\!\hspace*{0.2mm}H_{v}^{\infty}}$ is isometrically isomorphic to $\ell^1$ by \cite[Theorem 1.1]{4} since $v$ is normal, and thus there is a bounded bijective operator $U \colon {^*\!\hspace*{0.2mm}H_{v}^{\infty}} \to \ell^1$. The weakly compact operator $U \circ S \colon {^*\mathcal{Y}} \to \ell^1$ is compact since $\ell^1$ has the Schur property, which implies that $S = U^{-1} \circ (U \circ S)$ is compact, and from this we finally see that $T = \Phi_{\mathcal{Y}}^{-1}\circ S^* \circ \Phi_{H_{v}^{\infty}}$ is compact. Part (ii) can be proven using a similar argument as in the proof of part (i). See also the proof of \cite[Theorem 5.1]{5}.
\end{proof}

Since every bounded operator $T \colon \ell^{\infty} \to \mathcal{X}$ mapping into a Banach space $\mathcal{X}$ not containing a copy of $\ell^{\infty}$\! is weakly compact \cite{R}, we obtain the following result.

\begin{cor}
Let $\mathcal{Y} \subset H(\mathbb{D})$ be a Banach space satisfying condition \textnormal{(I)} and not fixing a copy of $\ell^\infty$\!. If $v$ is a normal weight, then the following statements are equivalent\textnormal{:}
\begin{itemize}
\item[(i)] $T_g^{\varphi} : H_{v}^{\infty} \to \mathcal{Y}$ is bounded.
\item[(ii)] $T_g^{\varphi} : H_{v}^{\infty} \to \mathcal{Y}$ is weakly compact.
\item[(iii)] $T_g^{\varphi} : H_{v}^{\infty} \to \mathcal{Y}$ is compact.
\end{itemize}
The same statement holds when $H_{v}^{\infty}$ is replaced with the Bloch space $\mathcal{B}_v^{\infty}$.
\end{cor}

The next theorem gives a nice estimate of the essential norm of $T_g^{\varphi}: \mathcal X \to H_v^\infty.$ As we have seen in section 2, the situation changes dramatically if the target space is instead $H^\infty$.

\begin{thm}\label{20}
Let $\mathcal{X} \subset H(\mathbb{D})$ be a Banach space satisfying conditions \textnormal{(I)-(V)}.
\begin{itemize}
\item[(i)]  If the weight $v$ is normal, then
\begin{equation*}
\lVert T_g^\varphi \rVert_{e, \mathcal{X} \to H_v^\infty} \asymp \limsup_{|\varphi(z)| \to 1} (1 - |z|)v(z)|(g \circ \varphi)'(z)| \lVert \delta_{\varphi(z)} \rVert_{\mathcal{X} \to\,\mathbb{C}}.
\end{equation*}
\end{itemize}
\begin{itemize}
\item[(ii)] 
For any weight $v$,
\begin{equation*}
\lVert T_g^\varphi \rVert_{e, \mathcal{X} \to \mathcal{B}_v^\infty} \asymp \limsup_{|\varphi(z)| \to 1} v(z)|(g \circ \varphi)'(z)| \lVert \delta_{\varphi(z)} \rVert_{\mathcal{X} \to\,\mathbb{C}}.
\end{equation*}
\end{itemize}
\end{thm}
\begin{proof}
To prove (i), define $w(z) := (1- |z|)v(z)$ and let $D \colon H_v^\infty \to H_{w}^\infty$ and $I \colon H_{w}^\infty \to H_v^\infty$ be the differentiation and integral operators used in the proof of Theorem \ref{24}, which are bounded due to the normality of $v$. For any compact operator $K \colon \mathcal{X} \to H_v^\infty$ we have
\begin{equation*}
\lVert D \circ T_g^\varphi \rVert_{e, \mathcal{X} \to H_w^\infty} \leq \lVert D \circ T_g^\varphi  - D \circ K \rVert_{\mathcal{X} \to H_w^\infty} \leq \lVert D \rVert_{H_v^\infty \to H_w^\infty} \lVert T_g^\varphi  - K \rVert_{\mathcal{X} \to H_v^\infty},
\end{equation*}
and therefore $\lVert D \circ T_g^\varphi  \rVert_{e,\mathcal{X} \to H_w^\infty} \leq \lVert D \rVert_{{H_v^\infty \to H_w^\infty}} \lVert T_g^\varphi  \rVert_{e,\mathcal{X} \to H_v^\infty}$. On the other hand, for all compact operators $K\colon \mathcal{X} \to H_{w}^\infty$ we have
\begin{align*}
\lVert T_g^\varphi  \rVert_{e,\mathcal{X} \to H_v^\infty} &\leq \lVert T_g^\varphi  - I \circ K \rVert_{\mathcal{X} \to H_v^\infty} = \lVert I \circ D \circ T_g^\varphi + \delta_0 \circ T_g^{\varphi} - I \circ K \rVert_{\mathcal{X} \to H_v^\infty} \\
&\leq \Vert I \Vert_{H_w^\infty \to H_v^\infty} \lVert D \circ T_g^\varphi  - K \rVert_{\mathcal{X} \to H_w^\infty},
\end{align*}
where we used that $I\circ D(h) = h - h(0)$ and $T_g^{\varphi}(f)(0) = 0$ for every $h \in H_v^{\infty}$ and $f \in \mathcal{X}$. This shows that $\lVert T_g^\varphi  \rVert_{e,\mathcal{X} \to H_v^\infty} \leq \lVert I \rVert_{H_v^\infty \to H_w^\infty} \lVert D \circ T_g^\varphi  \rVert_{e,\mathcal{X} \to H_w^\infty}$, and hence
\begin{align*}
\lVert T_g^\varphi  \rVert_{e,\mathcal{X} \to H_v^\infty} &\asymp \lVert D \circ T_g^\varphi  \rVert_{e,\mathcal{X} \to H_w^\infty} = \lVert (g \circ \varphi)' C_\varphi \rVert_{e,\mathcal{X} \to H_w^\infty}\\
&\asymp \limsup_{|\varphi(z)| \to 1} (1 - |z|)v(z)|(g \circ \varphi)'(z)| \lVert \delta_{\varphi(z)} \rVert_{\mathcal{X} \to\,\mathbb{C}}
\end{align*}
by \cite[Theorem 4.3]{5}. The proof of (ii) is similar, the only difference being that the operators $D \colon \mathcal{B}_v^{\infty} \to H_v^{\infty}$ and $I \colon H_v^{\infty} \to \mathcal{B}_v^{\infty}$ are bounded for any weight $v$.
\end{proof}

\begin{cor}\label{23}
Let $\alpha > -1$ and $1 \leq p < \infty$. If the weight $v$ is normal, then
\begin{itemize}
\item[(i)] $\lVert T_g^\varphi \rVert_{e,H^p \to H_v^\infty} \asymp \limsup_{|\varphi(z)| \to 1} \frac{1-|z|}{\left(1-|\varphi(z)|\right)^{1/p}}v(z)|(g \circ \varphi)'(z)|.$
\item[(ii)] $\lVert T_g^\varphi \rVert_{e,A_{\alpha}^p \to H_v^\infty} \asymp \limsup_{|\varphi(z)| \to 1} \frac{1-|z|}{\left(1-|\varphi(z)|\right)^{(2+\alpha)/p}}v(z)|(g \circ \varphi)'(z)|.$
\end{itemize}
Also, for any weight $v$
\begin{itemize}
\item[(iii)] $\lVert T_g^\varphi \rVert_{e,H^p \to \mathcal{B}_v^\infty} \asymp \limsup_{|\varphi(z)| \to 1}\frac{v(z)}{(1-|\varphi(z)|)^{1/p}}|(g \circ \varphi)'(z)|.$
\item[(iv)] $\lVert T_g^\varphi \rVert_{e,A_{\alpha}^p \to \mathcal{B}_v^\infty} \asymp \limsup_{|\varphi(z)| \to 1} \frac{v(z)}{\left(1-|\varphi(z)|\right)^{(2+\alpha)/p}}|(g \circ \varphi)'(z)|.$
\end{itemize}
\end{cor}

\begin{exam}
Let $1\leq p < \infty$ and $p-2<\alpha <\infty$. Corollary \ref{23} implies that the Volterra operator $T_g \colon A_p^\alpha \to H^\infty_{\frac{2+\alpha}{p}-1}$ is compact if and only if $\lim_{|z|\to 1}|g'(z)|=0$, or equivalently if and only if $g$ is a constant function. This means that the Volterra operator $T_g \colon A_p^\alpha \to H^\infty_{\frac{2+\alpha}{p}-1} = \mathcal{B}^\infty_{\frac{2+\alpha}{p}}$ is compact if and only if $T_g=0$.
\end{exam}

The following two lemmas will be needed to prove Theorem \ref{thm:22H} below, which uses the obtained essential norm estimates to relate compactness of $T_g^{\varphi} \colon \mathcal{X} \to H_v^{\infty}$ to compactness of $T_g^{\varphi} \colon \mathcal{X} \to H_v^0$, and also includes a corresponding comparison for the target spaces $\mathcal{B}_v^\infty$ and $\mathcal{B}_v^0$. 

\begin{lem}\label{22}
Let $\varphi$ be an analytic selfmap of $\mathbb{D}$ and $g \in H(\mathbb{D})$.
\begin{itemize}
\item[(i)] If $v$ is a normal weight and $g \circ \varphi \in H_v^{0}$, then $T_g^{\varphi}(f_r) \in H_v^{0}$ for every $f \in H(\mathbb{D})$ and $0 < r < 1$.
\item[(ii)] For any weight $v$, if $g \circ \varphi \in \mathcal{B}_v^{0}$ then $T_g^{\varphi}(f_r) \in \mathcal{B}_v^{0}$ for every $f \in H(\mathbb{D})$ and $0 < r < 1$.
\end{itemize}
\end{lem}
\begin{proof}
To prove (i), recall that $H_v^{0} = \mathcal{B}_{(1-r)v}^{0}$ since $v$ is normal. By assumption $(g \circ \varphi)^{\prime} \in H_{(1-r)v}^{0}$, and since $f_r \circ \varphi  \in H^{\infty}$ we see that
\begin{equation*}
(T_g^{\varphi}(f_r))^{\prime} = (f_r \circ \varphi)(g \circ \varphi)^{\prime} \in H_{(1-r)v}^{0},
\end{equation*}
and we are done. For part (ii), notice that for every $z \in \mathbb{D}$
\begin{equation*}
v(z)|(T_g^{\varphi}(f_r))^{\prime}(z)| = v(z)|f_r(\varphi(z))g^{\prime}(\varphi(z))\varphi^{\prime}(z)|  \leq \|f_r\|_{\infty}v(z)|(g \circ \varphi)^{\prime}(z)|,
\end{equation*}
and use $g \circ \varphi \in \mathcal{B}_v^{0}$ to obtain $T_g^{\varphi}(f_r) \in \mathcal{B}_v^{0}$.
\end{proof}

\begin{lem}\label{21}
Let $\varphi$ be an analytic selfmap of $\mathbb{D}$ and $g \in H(\mathbb{D})$. If the weight $v$ is normal and 
\begin{equation*}
\limsup_{|\varphi(z)| \to 1}(1-|z|)v(z)|(g \circ \varphi)^{\prime}(z)| = 0,
\end{equation*}
then 
\begin{equation*}
\limsup_{|z| \to 1}(1-|z|)v(z)|(g \circ \varphi)^{\prime}(z)| = 0.
\end{equation*}
\end{lem}
\begin{proof}
Choose a sequence $\{z_n\}_{n=1}^{\infty} \subset \mathbb{D}$ such that $|z_n| \to 1$ and
\begin{equation*}
\limsup_{|z| \to 1}(1-|z|)v(z)|(g \circ \varphi)^{\prime}(z)| = \lim_{n \to \infty}(1-|z_n|)v(z_n)|(g \circ \varphi)^{\prime}(z_n)|.
\end{equation*}
Since the sequence $\{\varphi(z_n)\}_{n=1}^{\infty}$ is bounded it has a convergent subsequence $\{\varphi(z_{n_k})\}_{k=1}^{\infty}$ with limit $y \in \overline{\mathbb{D}}$. If $y \in \partial \mathbb{D}$, then $|\varphi(z_{n_k})| \to 1$ and hence for every $0 < s < 1$ there is a number $K_s \in \mathbb{N}$ such that $|\varphi(z_{n_k})| > s$ whenever $k \geq K_s$, and the numbers $K_s$ can be chosen so that $\lim_{s \to 1}K_s = +\infty$. Thus for every $0 < s < 1$ it holds that
\begin{equation*}
\sup_{k \geq K_s}(1-|z_{n_k}|)v(z_{n_k})|(g \circ \varphi)^{\prime}(z_{n_k})| \leq \sup_{|\varphi(z)| > s}(1-|z|)v(z)|(g \circ \varphi)^{\prime}(z)|,
\end{equation*}
and after taking limits as $s \to 1$ we obtain
\begin{equation*}
\limsup_{|z| \to 1}(1-|z|)v(z)|(g \circ \varphi)^{\prime}(z)| \leq \limsup_{|\varphi(z)| \to 1}(1-|z|)v(z)|(g \circ \varphi)^{\prime}(z)| = 0,
\end{equation*}
and we are done. 

If on the other hand $y \in \mathbb{D}$ then $\lim_{k \to \infty}g^{\prime}(\varphi(z_{n_k})) = g^{\prime}(y)$. Since furthermore $v$ is normal we have that $\varphi \in H^{\infty} \subset H_v^{0} = \mathcal{B}_{(1-r)v}^{0}$, and thus 
\begin{equation*}
\lim_{k \to \infty}(1-|z_{n_k}|)v(z_{n_k})|\varphi^{\prime}(z_{n_k})| = 0,
\end{equation*}
from which follows that
\begin{align*}
\limsup_{|z| \to 1}(1-|z|)v(z)|(g \circ \varphi)^{\prime}(z)| &= \lim_{k \to \infty}(1-|z_{n_k}|)v(z_{n_k})|(g \circ \varphi)^{\prime}(z_{n_k})| \\
&= \lim_{k \to \infty}(1-|z_{n_k}|)v(z_{n_k})|\varphi^{\prime}(z_{n_k})||g^{\prime}(\varphi(z_{n_k}))| \\
&= 0 \cdot |g^{\prime}(y)| = 0,
\end{align*}
and the proof is complete.
\end{proof}

\begin{thm}\label{thm:22H}
Let $\mathcal{X} \subset H(\mathbb{D})$ be a Banach space.   
\begin{itemize}
\item[(i)] If the space $\mathcal{X}$ satisfies conditions \textnormal{(I)-(V)} and the weight $v$ is normal, then $T_g^{\varphi} \colon \mathcal{X} \to H_v^{\infty}$ is compact if and only if $T_g^{\varphi} \colon \mathcal{X} \to H_v^0$ is compact.
\item[(ii)] If the space $\mathcal{X}$ satisfies conditions \textnormal{(I)} and \textnormal{(IV)}, then for any weight $v$, $T_g^{\varphi} \colon \mathcal{X} \to \mathcal{B}_v^0$ is compact if and only if $T_g^{\varphi} \colon \mathcal{X} \to \mathcal{B}_v^\infty$ is compact and $g\circ \varphi \in \mathcal{B}_v^0$.
\end{itemize}
\end{thm}
\begin{proof}
To prove (i), assume that $T_g^\varphi\colon \mathcal{X} \to H_v^\infty$ is compact. Using Theorem \ref{20} and condition (II) we deduce that
\begin{equation*}
\limsup_{|\varphi(z)| \to 1}\,(1 - |z|)v(z)|(g \circ \varphi)'(z)| = 0,
\end{equation*}
and applying Lemma \ref{21} we arrive at
\begin{equation*}
\limsup_{|z| \to 1}(1-|z|)v(z)|(g \circ \varphi)^{\prime}(z)| = 0,
\end{equation*}
which shows that $g \circ \varphi \in \mathcal{B}_{(1-r)v}^{0} = H_v^{0}$. Now choose an arbitrary $f \in \mathcal{X}$ and note that $f_r \xrightarrow{co} f$ as $r \to 1^{-}$. Since $T_g^{\varphi}$ is compact and the set $\{f_r \colon {0<r<1}\}$ is bounded in $\mathcal{X}$ by (IV), we have by Lemma \ref{2} that 
\begin{equation*}
T_g^{\varphi}(f_r) \xrightarrow{\lVert \cdot \rVert_v} T_g^{\varphi}(f), \textnormal{ as } r \to 1^{-}.
\end{equation*} 
But $T_g^{\varphi}(f_r) \in H_v^0$ for every $f \in H(\mathbb{D})$ and $0 < r < 1$  by Lemma \ref{22} and therefore we must have that $T_g^{\varphi}(f) \in H_v^0$. This shows that $T_g^{\varphi}(\mathcal{X}) \subset H_v^0$ and the proof of (i) is complete. The proof of part (ii) is similar to the proof of part (i).
\end{proof}

The following example shows that in general for normal weights $v$ and $w$, compactness of $T_g^\varphi \colon \mathcal{B}_v^\infty \to \mathcal{B}_w^\infty$ does not imply that $g \circ \varphi \in \mathcal{B}_w^0$.
\begin{exam}
Note that for each $0<\alpha <1$, the Volterra operator $T_g \colon \mathcal{B}_{v_{\alpha}}^\infty \to \mathcal{B}_{v_{\alpha}}^\infty$ is compact if and only if the multiplication operator $M_{g'} \colon \mathcal{B}_{v_{\alpha}}^\infty \to H_{v_{\alpha}}^\infty = \mathcal{B}_{v_{\alpha+1}}^\infty$ is compact, and this happens precisely when $g\in \mathcal{B}_{v_1}^0$ by \cite[Theorem 3.1]{10}. Let $0<\alpha <1$ and $0<\delta <1-\alpha$ and define $g_\delta (z) := (1-z)^\delta$ for $z\in \mathbb{D}$. Then, one can see that 
\begin{equation*}
|g_\delta'(z)|(1-|z|^2)\leq 2\delta (1-|z|^2)^\delta
\end{equation*}
for every $z\in \mathbb{D}$. This implies that $g_\delta\in \mathcal{B}_{v_1}^0$ and therefore $T_{g_\delta} \colon \mathcal{B}_{v_{\alpha}}^\infty \to \mathcal{B}_{v_{\alpha}}^\infty$ is compact.
Also, note that for each $r\in (0,1)$ we have
\begin{equation*}
|g_\delta'(r)|(1-|r|^2)^\alpha = \frac{\delta (1-r)^\alpha (1+r)^\alpha}{(1-r)^{1-\delta}}\geq \delta (1-r)^{\alpha + \delta -1},
\end{equation*}
which implies that $g_\delta\notin \mathcal{B}_{v_{\alpha}}^0$.
\end{exam}

We also have a natural weak compactness version of Theorem \ref{thm:22H} above.

\begin{thm}\label{thm:TgwXH}
Let $\mathcal{X} \subset H(\mathbb{D})$ be a Banach space satisfying conditions \textnormal{(I)} and \textnormal{(IV)}.
\begin{itemize}
\item[(i)] If $v$ is a normal weight, then $T_g^{\varphi} \colon \mathcal{X} \to H_v^0$ is weakly compact if and only if $T_g^{\varphi} \colon \mathcal{X} \to H_v^{\infty}$ is weakly compact and $g \circ \varphi \in H_v^{0}$.
\item[(ii)] If $v$ is a normal weight and $H_{v_0}^\infty \subset \mathcal{X}$ for some normal weight ${v_0}$, then the operator $T_g^{\varphi}\colon \mathcal{X} \to H_v^\infty$ is weakly compact if and only if $T_g^{\varphi} \colon \mathcal{X} \to H_v^0$ is weakly compact.
\item[(iii)] For any weight $v$, $T_g^{\varphi} \colon \mathcal{X} \to B_v^0$ is weakly compact if and only if $T_g^{\varphi} \colon \mathcal{X} \to B_v^{\infty}$ is weakly compact and $g \circ \varphi \in B_v^{0}$.
\end{itemize}
\end{thm}
\begin{proof}
We begin by proving (i). If $T_g^{\varphi} \colon \mathcal{X} \to H_v^0$ is weakly compact then obviously $T_g^{\varphi} \colon \mathcal{X} \to H_v^{\infty}$ is weakly compact and also $g \circ \varphi \in H_v^{0}$ since $\mathcal{X}$ contains the constant functions. Conversely, assume that $T_g^{\varphi} \colon \mathcal{X} \to H_v^{\infty}$ is weakly compact and $g \circ \varphi \in H_v^{0}$. Let $f\in \mathcal{X}$ and choose an arbitrary sequence $\{r_n\}_{n=1}^{\infty} \subset (0,1)$ such that $r_n \rightarrow 1$. Then $f_{r_n} \xrightarrow{co} f$ and since $\sup_{n \in \mathbb{N}}\|f_{r_n}\|_{\mathcal{X}}<\infty$ by (IV) we have that $T_g^{\varphi}f_{r_n} \xrightarrow{w} T_g^{\varphi}f$ by Lemma \ref{2}. On the other hand, by Lemma \ref{22} (i), we know that $T_g^{\varphi}f_{r_n}\in H_v^{0}$ for each $n\in \mathbb{N}$. This implies that $T_g^{\varphi}f\in H_v^{0}$, showing that $T_g^{\varphi}(\mathcal{X}) \subset H_v^{0}$ and the proof of (i) is complete.

To prove (ii), note that since $H_{v_0}^\infty \subset \mathcal{X}$, condition (I) implies that the identity operator $\textnormal{id} \colon H_{v_0}^\infty \to \mathcal{X}$ is bounded. Thus if $T_g^\varphi \colon \mathcal{X} \to H_v^\infty$ is weakly compact then $T_g^\varphi \colon H_{v_0}^\infty \to H_v^\infty$ is weakly compact. Hence, by Corollary \ref{cor:HH}, $g\circ \varphi \in H_v^0$ and therefore Theorem \ref{thm:TgwXH} (i) implies that $T_g^\varphi \colon \mathcal{X} \to H_v^0$ is weakly compact, proving (ii). The proof of part (iii) is similar to the proof of part (i). 
\end{proof}

\begin{rem}
Theorem \ref{thm:TgwXH} (ii) can for example be applied to the weighted Bergman spaces $\mathcal{X} = A_{\alpha}^p$ with $1 \leq p < \infty$ and $\alpha > 0$ because $H_{v_{\!\hspace*{0.2mm}\frac{\alpha}{p}}}^{\infty} \subset A_{\alpha}^p$.
\end{rem}

We end this paper with two corollaries which generalize \cite[Theorem 2]{14} and summarize some of the theorems obtained in this section, but first we need an additional result. The closed unit ball of $H^0_v$ is $co$-dense in the closed unit ball $B_{H^\infty_v}$ for any weight $v$, so the restriction map $\varrho \colon {}^*\!\hspace*{0.2mm}H_v^\infty \to (H_v^0)^*$, mapping $\ell \mapsto \ell|H^0_v$, is an isometric isomorphism by \cite[Theorem 1.1]{11}. Using the evaluation map $\Phi_{H_v^{\infty}} \colon H_v^{\infty} \to (^{\ast}\!\hspace*{0.2mm}H_v^{\infty})^{\ast}$ one then obtains an isometric isomorphism $\Lambda \colon H_v^{\infty} \to (H_v^0)^{\ast\ast} $ by defining 
\begin{equation*}
\Lambda(f) := \left(\widehat{f}^{\,H_v^{\infty}}\big|{^{\ast}\!\hspace*{0.2mm}H_v^{\infty}}\right) \circ \varrho^{-1}.
\end{equation*}
Here we use the notation $\widehat{f}^{\,\mathcal{X}}$ to emphasize which space $\mathcal{X}$ the evaluation map is acting on, that is, $\widehat{f}^{\,\mathcal{X}} \in \mathcal{X}^{\ast\ast}$ whenever $f \in \mathcal{X}$. Using the differentiation operator $D \colon \widetilde{\mathcal{B}}^\infty_v \to H^\infty_v$ and the integration operator $I \colon H^\infty_v \to \widetilde{\mathcal{B}}^\infty_v$ as natural isometric isomorphisms between the spaces $\widetilde{\mathcal{B}}^\infty_v$ and $H^\infty_v$, one obtains a corresponding isometric isomorphism $\Delta \colon \widetilde{\mathcal{B}}^\infty_v \to \big(\widetilde{\mathcal{B}}_v^0\big)^{**}$.

\begin{lem}\label{lem:16}
Let $v$ and $w$ be weights and $T \colon H(\mathbb{D}) \to H(\mathbb{D})$ be a co-co-continuous operator.
\begin{itemize}
\item[(i)]  If $T \colon H_v^0 \to H_w^\infty$ is weakly compact, then for every $f \in H_v^{\infty}$
\begin{equation*}
T^{**}(\Lambda(f)) = \widehat{T(f)}^{H_w^{\infty}}\!.
\end{equation*}
\item[(ii)] If $T \colon \widetilde{\mathcal{B}}_v^0 \to \widetilde{\mathcal{B}}_w^\infty$ is weakly compact, then for every $f \in \widetilde{\mathcal{B}}_v^{\infty}$
\begin{equation*}
	T^{**}(\Delta(f)) = \widehat{T(f)}^{\widetilde{\mathcal{B}}_w^{\infty}}\!.
\end{equation*}
\end{itemize}
\end{lem}
\begin{proof}
To prove (i), assume that  $T \colon H_v^0 \to H_w^\infty$ is weakly compact, choose a function $f \in H_v^{\infty}$ not identically zero and define $g := \frac{f}{\|f\|_{H_v^{\infty}}} \in B_{H_v^{\infty}}$. Note that the operator $T \colon H_v^{\infty} \to H_w^{\infty}$ is bounded by Theorem \ref{thm:bdd1} (see Remark \ref{27}), so the equality about to be proven makes sense. Since the mapping $\Lambda \colon H_v^\infty \to ({H_v^0})^{**}$ is an isometric isomorphism, we have by Goldstine's theorem that $\Lambda(g) \in B_{({H_v^0})^{**}} = \overline{B_{\widehat{{H}_v^0}}}^{\,w^{\ast}}$. Hence, there exists a net $\left\{\widehat{g}_{\gamma}^{\,H_v^0}\right\} \subset B_{\widehat{{H}_v^0}}$, where every $g_{\gamma} \in B_{H_v^0}$, such that $\widehat{g}_{\gamma}^{\,H_v^0} \xrightarrow{w^*} \Lambda(g)$ in ${(H_v^0})^{**}$. Moreover, since $\varrho^{-1}(\ell)|H_v^{0} = \ell$ for every $\ell \in (H_v^0)^{\ast}$, we have that
\begin{equation*}
\Lambda(g_{\gamma}) = \left(\widehat{g}_{\gamma}^{\,H_v^{\infty}}\big|{^{\ast}\!\hspace*{0.2mm}H_v^{\infty}}\right) \circ \varrho^{-1} = \widehat{g}_{\gamma}^{\,H_v^0},
\end{equation*}
and thus $\Lambda(g_{\gamma}) \xrightarrow{w^*} \Lambda(g)$ in ${(H_v^0})^{**}$. For any $z \in \mathbb{D}$, $\delta_z \in {^{\ast}\!\hspace*{0.2mm}H_v^{\infty}}$ and $\varrho(\delta_z) \in  (H_v^0)^{\ast}$, which gives that
\begin{equation*}
g_{\gamma}(z) = \Lambda(g_{\gamma})(\varrho(\delta_z)) \rightarrow \Lambda(g)(\varrho(\delta_z)) = g(z).
\end{equation*}
Furthermore, $\{g_{\gamma}\}$ is an equicontinuous family because $\|g_{\gamma}\|_{H_v^{\infty}} \leq 1$ for every $\gamma$, so we actually have that $g_{\gamma} \xrightarrow{co} g$, and since $T$ is \textit{co}-\textit{co}-continuous we get $T(g_{\gamma}) \xrightarrow{co} T(g)$.

On the other hand, since $T \colon H_v^0 \to H_w^\infty$ is weakly compact, $T^{\ast} \colon (H_w^{\infty})^{\ast} \to (H_v^0)^{\ast}$ is also weakly compact, and hence $T^{\ast\ast} \colon (H_v^0)^{\ast\ast} \to (H_w^\infty)^{\ast\ast}$ is $w^{\ast}$-$w$-continuous. We thus have that
\begin{equation}\label{26}
\widehat{T(g_{\gamma})}^{H_w^{\infty}} = T^{**}\big(\widehat{g}_{\gamma}^{H_v^0}\big) \xrightarrow{w} T^{**}(\Lambda(g)),
\end{equation}
where the first equality holds because
\begin{equation*}
\widehat{T(g_{\gamma})}^{H_w^{\infty}}(\ell) = \ell(T(g_{\gamma})) = T^{\ast}(\ell)(g_{\gamma}) = \widehat{g}_{\gamma}^{H_v^0}(T^{\ast}(\ell)) = T^{**}\big(\widehat{g}_{\gamma}^{H_v^0}\big)(\ell)
\end{equation*}
for any $\ell \in (H_w^{\infty})^{\ast}$. The left-hand side of (\ref{26}) is contained in $\widehat{H_w^\infty}$ for every $\gamma$, which gives that
\begin{equation*}
T^{**}(\Lambda(g)) \in \overline{\widehat{H_w^\infty}}^{\,w} = \overline{\widehat{H_w^\infty}}^{\,\|\cdot\|_{(H_w^{\infty})^{\ast\ast}}} = \widehat{H_w^\infty},
\end{equation*}
and hence there exists $h\in H_w^\infty$ such that  $T^{**}(\Lambda(g)) = \widehat{h}^{H_w^{\infty}}$. Now since obviously
\begin{equation*}
\widehat{T(g_{\gamma})}^{H_w^{\infty}}  \xrightarrow{w^*} \widehat{h}^{H_w^{\infty}}
\end{equation*}
and $\delta_z \in (H_w^{\infty})^{\ast}$ for any $z \in \mathbb{D}$, it follows that $T(g_{\gamma})$ converges pointwise to $h$ on the unit disc. Hence, we must have that $h = T(g)$, showing that
\begin{equation*}
T^{**}(\Lambda(g)) = \widehat{T(g)}^{H_w^{\infty}},
\end{equation*}
which also holds when $g$ is replaced by $f = \|f\|_{H_v^{\infty}}g$ due to linearity, and the proof of (i) is complete. Part (ii) is proved in a similar way.
\end{proof}

\begin{cor}\label{cor:HH}
Let $v$ and $w$ be normal weights and assume that $v$ is equivalent to its associated weight $\widetilde{v}$. Then the following statements are equivalent\textnormal{:}
\begin{itemize}
\item[(1)] $T_g^\varphi \colon H_v^\infty \to H_w^\infty$ is compact.
\item[(2)] $T_g^\varphi \colon H_v^\infty \to H_w^\infty$ is weakly compact.
\item[(3)] $T_g^\varphi \colon H_v^\infty \to H_w^0$ is compact.
\item[(4)] $T_g^\varphi \colon H_v^\infty \to H_w^0$ is weakly compact or equivalently bounded.
\item[(5)] $T_g^\varphi \colon H_v^0 \to H_w^0$ is compact.
\item[(6)] $T_g^\varphi \colon H_v^0 \to H_w^0$ is weakly compact.
\item[(7)] $T_g^\varphi \colon H_v^0 \to H_w^\infty$ is compact.
\item[(8)] $T_g^\varphi \colon H_v^0 \to H_w^\infty$ is weakly compact.
\item[(9)] $\lim_{|\varphi(z)| \to 1}(1-|z|)|(g \circ \varphi)^{\prime}(z)|\dfrac{w(z)}{v(\varphi(z))}  = 0$.
\end{itemize}
\end{cor}
\begin{proof}
We begin by justifying the equivalence within statement (4). If $T \colon H_v^{\infty} \to H_w^0$ is any bounded operator and $\{\ell_n\}_{n=1}^{\infty} \subset B_{(H_w^0)^{\ast}}$, then there is a $w^{\ast}$-convergent subsequence $\{\ell_{n_k}\}_{k=1}^{\infty}$ with some limit $\ell \in  B_{(H_w^0)^{\ast}}$. This follows from the Alaoglu theorem, since the topology $(B_{(H_w^0)^{\ast}}, \sigma((H_w^0)^{\ast}, H_w^0))$ is metrizable due to the separability of $H_w^0$. This gives that $T^{\ast}(\ell_{n_k}) \xrightarrow{w^{\ast}} T^{\ast}(\ell)$, but $H_v^{\infty} \approx \ell^{\infty}$ is a Grothendieck space meaning that we actually have weak convergence of $T^{\ast}(\ell_{n_k})$ to $T^{\ast}(\ell)$, and hence both $T^{\ast} \colon (H_w^0)^{\ast} \to (H_v^{\infty})^{\ast}$ and $T \colon H_v^{\infty} \to H_w^0$ are weakly compact. 

The statements (1) and (2) are equivalent by Theorem \ref{13} (i) and the equivalence between (1) and (3) follows from Theorem \ref{thm:22H} (i). If the operator $T_g^\varphi \colon H_v^0 \to H_w^\infty$ is weakly compact, then  ${(T_g^\varphi)}^{**} \colon  (H_v^0)^{**} \rightarrow \widehat{H_w^\infty}$ is weakly compact, and hence, if $Q \colon H_w^{\infty} \to \widehat{H_w^\infty}$ is the canonical embedding $Q(f) = \widehat{f}$, we have by Lemma \ref{lem:16} (i) that
\begin{equation*}
Q^{-1}\circ {(T_g^\varphi)}^{**} \circ \Lambda = T_g^\varphi \colon H_v^\infty \to H_w^\infty
\end{equation*}
 is weakly compact, and therefore (8) implies (2). Finally, by Theorem \ref{20} (i), (1) is equivalent to (9), and the rest of the implications are obvious.
\end{proof}

\begin{cor}
Let $v$ and $w$ be normal weights. Then the following statements are equivalent\textnormal{:}
\begin{itemize}
\item[(1)] $T_g^\varphi \colon \widetilde{\mathcal{B}}_v^\infty \to \widetilde{\mathcal{B}}_w^\infty$ is compact and $g\circ \varphi \in \widetilde{\mathcal{B}}_w^0$.
\item[(2)] $T_g^\varphi \colon \widetilde{\mathcal{B}}_v^\infty \to \widetilde{\mathcal{B}}_w^\infty$ is weakly compact and $g\circ \varphi \in \widetilde{\mathcal{B}}_w^0$.
\item[(3)] $T_g^\varphi \colon \widetilde{\mathcal{B}}_v^\infty \to \widetilde{\mathcal{B}}_w^0$ is compact.
\item[(4)] $T_g^\varphi \colon \widetilde{\mathcal{B}}_v^\infty \to \widetilde{\mathcal{B}}_w^0$ is weakly compact or equivalently bounded.
\item[(5)] $T_g^\varphi \colon \widetilde{\mathcal{B}}_v^0 \to \widetilde{\mathcal{B}}_w^0$ is compact.
\item[(6)] $T_g^\varphi \colon \widetilde{\mathcal{B}}_v^0 \to \widetilde{\mathcal{B}}_w^0$ is weakly compact.
\item[(7)] $T_g^\varphi \colon \widetilde{\mathcal{B}}_v^0 \to \widetilde{\mathcal{B}}_w^\infty$ is compact  and $g\circ \varphi \in \widetilde{\mathcal{B}}_w^0$.
\item[(8)] $T_g^\varphi \colon \widetilde{\mathcal{B}}_v^0 \to \widetilde{\mathcal{B}}_w^\infty$ is weakly compact  and $g\circ \varphi \in \widetilde{\mathcal{B}}_w^0$.
\end{itemize}
\end{cor}
\begin{proof}
Similar reasoning as in the proof of Corollary \ref{cor:HH}. Notice that the Bloch space $\widetilde{\mathcal{B}}_v^\infty$ satisfies conditions (I) and (IV) required in Theorem \ref{thm:22H} (ii) for any normal weight $v$.
\end{proof}

\noindent \textbf{Acknowledgements.} The first and last authors are grateful for the financial support from the Doctoral Network in Information Technologies and Mathematics at \AA bo Akademi University. 
The second, fourth and last authors were  partially supported by the Academy of Finland project 296718. The second author is thankful to Wayne Smith for informing about reference \cite{15}. Part of this research was done while the third and fourth authors were visiting \AA bo Akademi University, whose hospitality is gratefully acknowledged.

\end{document}